\documentclass[reqno,12pt]{amsart}
\newcommand{\RR}{{\mathbb R}}
\newcommand{\CC}{{\mathbb C}}

\def\bege{\begin{equation}} \def\ende{\end{equation}}
\def\begr{\begin{eqnarray}} \def\endr{\end{eqnarray}}
\usepackage{amsmath}
\usepackage{amssymb}
\usepackage{cite}

\usepackage{bbm}
\usepackage{amsmath}
\usepackage{txfonts}
\usepackage{stmaryrd}
\usepackage{amssymb}
\usepackage{amsfonts}
\usepackage{times}
\usepackage{mathrsfs}

\usepackage{amsthm}
\usepackage{enumerate}

\usepackage{framed}
\usepackage{lipsum}
\usepackage{color}

\def\BB{ \mathbb{B}}

\def\CC{ \mathbb{C}}
\newcommand{\DD}{{\mathbb D}}

\def\R{\mathcal{R}}

\def\T{\mathcal{T}}

\def\D{\mathbb{D}}
\def\N{\mathbb N}

\def\hD{\hat{\mathcal{D}}}
\def\dD{\mathcal{D}}

\def\om{\omega}

\def\p{{\prime}}

\def\begr{\begin{eqnarray}} \def\endr{\end{eqnarray}}

\def\msk{\medskip}
\def\ol{\overline}
\newtheorem{Lemma}{Lemma}
\newtheorem{Theorem}{Theorem}

\begin{document}
\title[  Schatten-Herz class Toeplitz operators  ]{ Schatten-Herz class Toeplitz operators on weighted Bergman spaces induced by  doubling weights}

 \author{ Juntao Du and    Songxiao Li$\dagger$ }
 \address{Juntao Du\\ Faculty of Information Technology, Macau University of Science and Technology, Avenida Wai Long, Taipa, Macau.}
 \email{jtdu007@163.com  }

 \address{Songxiao Li\\ Institute of Fundamental and Frontier Sciences, University of Electronic Science and Technology of China,
 610054, Chengdu, Sichuan, P.R. China. } \email{jyulsx@163.com}

 \begin{abstract} Schatten-Herz class  Toeplitz operators on  weighted Bergman spaces induced by  doubling weights  are  investigated in this paper.
 \thanks{$\dagger$ Corresponding author.}
 \end{abstract}
 \maketitle

\section{Introduction}
Let $\D$ be the open unit disk in the complex plane, and $H(\D)$ be the class of all functions analytic on $\D$.
For any $z\in\D$ and $r>0$, let  $D(z,r)=\{w\in\D:\beta(z,w)<r\}$ be the Bergman disk.
Here  $\beta(\cdot,\cdot)$ is the Bergman metric  on $\DD$.
If $\{a_j\}_{j=1}^\infty\subset\D $ satisfying $\inf_{i\neq j}\beta(a_i,a_j)\geq s>0$, we say that $\{a_j\}_{j=1}^\infty\subset\D $ is $s$-separated.

A function  $\om: \D \rightarrow [0,\infty)$ is called a weight if it is  positive and integrable. $\om$ is radial if $\om(z)=\om(|z|)$ for all $z\in\DD$.  Let  $\om$ be a  radial weight and $\hat{\om}(r)=\int_r^1 \om(s)ds$ for $r\in [0,1)$.
We say that $\om$ is a doubling weight, denoted by $\om\in \hD$, if  there is a constant $C>0$ such that
$\hat{\om}(r)<C\hat{\om}(\frac{1+r}{2})$ when $   0\leq r<1.$ If there exist $K>1$ and $C>1$ such that
$\hat{\om}(r)\geq C\hat{\om}(1-\frac{1-r}{K}) $ when $  0\leq r<1$, we say that $\om$ is a reverse doubling weight,    denoted by $\om\in\check{\mathcal{D}}$.
We say that $\om$ is a regular weight, denoted by $\om\in\R$, if there exists $C>1$, such that
$$\frac{1}{C}<\frac{\hat{\om}(r)}{(1-r)\om(r)}<C, \,\,\mbox{ when }\,\,0\leq r<1 .$$
We denote by $\dD= \hD \bigcap\check{\mathcal{D}}$.   From \cite{PjaRj2016am}, we see that $\R\subset\dD$.
More details about $\R$, $\dD$ and $\hD$  can be found in \cite{Pja2015,PjaRj2014book,PjaRj2016am,PjaRjSk2018jga,PjRj2016jmpa}.

When $0<p<\infty$ and $\om\in\hD$,    the weighted Bergman space $A_\om^p$ is the space of all $f\in H(\D)$ such that
$$\|f\|_{A_\om^p}^p=\int_{\D}|f(z)|^p\om(z)dA(z)<\infty,$$
where $dA(z)$ is the normalized Lebesgue area measure on $\D$. When $\om(z)=(1-|z|^2)^\alpha(\alpha>-1)$, the space $A_\om^p$ becomes the classical weighted Bergman space $A_\alpha^p$. When $\alpha=0$, we will write $A_\alpha^p=A^p$. Suppose $0<p\leq \infty$ and $\mu$ is a positive Borel measure on $\DD$. Let $L_\mu^p$ denote the Lebesgue space defined in a standard way.

Let  $\om_s=\int_0^1 r^s\om(r)dr$ and $B_z^\om(w)=\frac{1}{2}\sum_{k=0}^\infty \frac{(w\ol{z})^k}{\om_{2k+1}} .$ Then $B_z^\om$ is the reproducing kernel for
$A_\om^2$, which means that for any $f\in A^2_\om$(see \cite{PjRj2016jmpa}),
\begin{align*}
f(z)=\langle f, B_z^\om \rangle_{A_\om^2}=\int_\BB f(w)\ol{B_z^\om(w)}\om(w)dA(w).
\end{align*}
Let $\mu$ be a positive Borel measure on $\D$. The Toeplitz operator  $\T_\mu$ and the Berezin transform $\widetilde{\T_\mu}$ of $\T_\mu$ are defined by
$$\T_\mu f (z)=\int_\BB f(\xi)\ol{B_z^\om(\xi)}d\mu(\xi), \,\,\,\,\,\,\,\,\,\,
\widetilde{\T_\mu}(z)=\frac{\langle \T_\mu B_z^\om, B_z^\om \rangle_{A_\om^2}}{\|B_z^\om\|_{A_\om^2}^2},\,\,\,\,z\in\DD,$$
respectively.

For $k=0,1,2,\cdots$, let $\Lambda_k=\{z\in\DD:1-\frac{1}{2^k}\leq|z|<1-\frac{1}{2^{k+1}}\}$  and $\chi_k$ be the characteristic function of $\Lambda_k$.
Let $d\lambda(z)=\frac{dA(z)}{(1-|z|^2)^2}$ be the M\"obious invariant area measure on $\DD$.
For $0<p\leq \infty$ and $0\leq q\leq \infty$, let $\mathrm{K}_q^p(\lambda)$ denote the Herz space, which consists of all measurable functions  $f$ such that
$$\|f\|_{\mathrm{K}_q^p(\lambda)}=\| \{\|f\|_{L^p_{\lambda\chi_k}}\}_{k=0}^\infty  \|_{l^q}<\infty.$$
Note that $\mathrm{K}_p^p(\lambda)=L^p_\lambda$.
Here, $l^q(0\leq q\leq \infty)$ consists of  all  complex sequences $\{a_k\}_{k=1}^\infty$ such that $\|\{a_k\}\|_{l^q}<\infty$, where
$$
\|\{a_k\}\|_{l^q}=\left\{
\begin{array}{cc}
  \left(\sum_{k=1}^\infty |a_k|^q\right)^\frac{1}{q}, & 0<q<\infty; \\
  \sup_{k\geq 1}|a_k|, & q=\infty; \\
  \limsup_{k\to\infty}|a_k|, & q=0.
\end{array}
\right.
$$

Let $j=0,1,2,\cdots$ and
$\lambda_j=\inf\{\|\T_\mu-R\|_{A_\om^2 \to A_\om^2}:\mbox{rank}(R)\leq j\}.$
If $\{\lambda_j\}_{k=0}^\infty\in l^p$ for some $p\in(0,\infty)$, we say that $\T_\mu$ belongs to the Schatten $p$-class, denoted by $\T_\mu\in \mathcal{S}_p(A_\om^2)$.
We denote by $\mathcal{S}_\infty$ the class of all bounded linear operators on $A_\om^2$.
For $0<p\leq\infty$ and $0\leq q\leq \infty$, $\T_\mu$ is said to belong to the Schatten-Herz class, denoted by  $\mathcal{S}_{p,q}$,     if $\T_{\mu\chi_k} \in \mathcal{S}_p$
and $$\|\T_\mu\|_{\mathcal{S}_{p,q}}=\|\{\|\T_{\mu\chi_k}\|_{\mathcal{S}_p}\}_{k=0}^\infty\|_{l^q}<\infty.$$
Here $\|\T_{\mu\chi_k}\|_{\mathcal{S}_p}$ means the Schatten $p$-class norm of $\T_{\mu\chi_k}$ on $A_\om^2$, see \cite{zhu} for example.
In \cite{LmLmPs2006ieot},  Loaiza,  L\'opez-Garc\'ia and   P\'erez-Esteva  considered Schatten-Herz class Toeplitz operators  on $A^2$  for the first time.
For more study on Herz spaces and Schatten-Herz class Toeplitz operators, see \cite{BoPs2011ieot, CbKhKn2007nmj, CeNk2007jmaa,HzTx2010pams,LmLmPs2006ieot, LxHz2010, LxHz2013mn}.

In \cite{PjaRj2016am,PjaRjSk2018jga}, Pel\'aez, R\"atty\"a and Sierra investigated the boundedness, compactness and Schatten class Toeplitz operators on Bergman spaces induced by doubling weights.
In this paper, we investigate  Schatten-Herz class Toeplitz operators on $A_\om^2$ with $\om\in\hD$.
The main result of this paper is stated as follows.

\begin{Theorem}\label{1103-1}
Suppose $0<p\leq \infty$, $0\leq q\leq \infty$, $r>0$, $\om\in\hD$ and $\mu$ is a positive Borel measure.
Let $\{a_j\}\subset\DD$ such that $\DD=\cup_{j=1}^\infty D(a_j,r)$ and $\{a_j\}$ are $s$-separated for some $0<s<r<\infty.$
Let
$$\widehat{\mu_r}(z)=\frac{\mu(D(z,r))}{(1-|z|)\hat{\om}(z)}\,\,\mbox{ and }\,\,
\|\{\widehat{\mu_r}(a_j)\}\|_{l_q^p}= \|\{\|\{(\widehat{\mu_r}\chi_k)(a_j)\}_{j=1}^\infty\|_{l^p}\}_{k=0}^\infty\|_{l^q}.$$
\begin{enumerate}[(i)]
  \item The  following statements are equivalent.
       \begin{itemize}
         \item [{\it (a)}] $\T_\mu\in\mathcal{S}_{p,q}$;
         \item [{\it (b)}]  $\widehat{\mu_r}\in\mathrm{K}_q^p(\lambda)$;
         \item [{\it (c)}]  $\|\{\widehat{\mu_r}(a_j)\}\|_{l_q^p}<\infty$.
       \end{itemize}
       Moreover,   $$\|\T_\mu\|_{\mathcal{S}_{p,q}}\approx\|\widehat{\mu_r}\|_{\mathrm{K}_q^p(\lambda)}\approx\|\{\widehat{\mu_r}(a_j)\}\|_{l^p_q}.$$
  \item If $\om\in\dD$ such that $\frac{(1-|\cdot|)^{p}\hat{\om}(\cdot)^p}{(1-|\cdot|)^{2}}\in\dD$,  then $\T_\mu\in\mathcal{S}_{p,q}$ if and only if $\widetilde{\T_\mu}\in\mathrm{K}_q^p(\lambda)$.
      \end{enumerate}
 Moreover,    $$\|\T_\mu\|_{\mathcal{S}_{p,q}}\approx \|\widetilde{\T_\mu}\|_{\mathrm{K}_q^p(\lambda)}.$$
\end{Theorem}

Throughout this paper, the letter $C$ will denote  constants and may differ from one occurrence to the other.
The notation $A \lesssim B$ means that there is a positive constant C such that $A\leq CB$.
The notation $A \approx B$ means $A\lesssim B$ and $B\lesssim A$.\msk

 \section{The proof of main result}
In this section, we prove the main result in this paper. For this purpose, let's recall some related definitions and state some lemmas.

For $j=0,1,2,\cdots$ and $k=0,1,2,\cdots,2^j-1$, let
$$I_{j,k}=\left\{e^{i\theta}:\frac{2\pi k}{2^j}\leq \theta<\frac{2\pi(k+1)}{2^j}\right\}$$
 and
 $$R_{j,k}=\left\{z\in\DD:\frac{z}{|z|}\in I_{j,k},1-\frac{|I|}{2\pi}\leq |z|<1-\frac{|I|}{4\pi}\right\}.$$

\begin{Lemma}\label{1105-1}
Suppose $0<s<r<\infty$ are given and $\{a_n\}_{n=1}^\infty$ is $s$-separated. Then the following statements hold.
\begin{enumerate}[(i)]
  \item For any $z\in\DD$, there exist at most $N_1=N_1(r)$ elements of $\{R_{j,k}\}$ intersect with $D(z,r)$;
  \item For any $R_{j,k}$, there exist at most $N_2=N_2(r,s)$ elements of $\{D(a_n,r)\}$ intersect with $R_{k,j}$.
\end{enumerate}
\end{Lemma}

\begin{proof}{\it (i).} Suppose $z\in\Lambda_i$ and $w\in D(z,r)$.
Then we have $1-|z|\approx1-|w|$. So there exists $M=M(r)$ such that $w\in\cup_{m=i-M}^{i+M}\Lambda_m$.
Here, let $\Lambda_m=\O$ if $m<0$.
By Proposition 4.4 in \cite{zhu}, $D(z,r)$ is an Euclidean disk with center $C_0=\frac{1-(\tanh r)^2}{1-(\tanh r)^2|z|^2}z$  and radius
$R_0=\frac{1-|z|^2}{1-(\tanh r)^2|z|^2}\tanh r.$
So, for all $w\in D(z,r)$, we have
$$|\mbox{Arg} (\ol{z}w)| \leq \arcsin\frac{R_0}{C_0}\approx 1-|z|,\,\, \mbox{as}\,\,|z|\to 1. $$
When $j=i-M,\cdots,i+M$ and $k=1,2,\cdots,2^j-1$, $|I_{j,k}|\approx \frac{1}{2^i}\approx 1-|z|$.
So, there exists $N_1=N_1(r)$ such that there are at most $N_1$ elements of $\{R_{j,k}\}$ intersect with $D(z,r)$.

{\it (ii).}
Let $c_{t,j,k}=(1-\frac{1}{2^t})e^{\frac{2\pi\mathrm{i} k}{2^j}}$. For any $w\in R_{j,k}$, as $j\to\infty$, we have
\begin{align*}
\beta(c_{j,j,k},w)\leq \beta(c_{j,j,k},c_{j+1,j,k})+\beta(c_{j,j,k},c_{j,j,k+1})\lesssim 1.
\end{align*}
So, there exists $R^\p>0$, for all $j=0,1,2,\cdots$, $k=0,1,\cdots,2^j-1$ and $w\in R_{j,k}$, we have $\beta(c_{j,j,k},w)<R^\p$.

  Given $R_{j,k}$. Without loss of generality, assume $R_{j,k}$ intersect with $D(a_i,r)$ for $i=1,2,\cdots,N_{j,k}$.
 We have
$$1-|c_{j,j,k}|\approx 1-|a_i|,\,\,i=1,2,\cdots,N_{j,k}$$
and
$$\cup_{i=1}^{N_{j,k}} D(a_i,s) \subset \cup_{i=1}^{N_{j,k}} D(a_i,r)\subset D(c_{j,j,k}, R^\p+2r).$$
Thus
$$N_{j,k}\leq \frac{|D(c_{j,j,k},R^\p+2r)|}{\inf_{1\leq i\leq N_{j,k}}|D(a_i,s)|}\lesssim 1,$$
which implies the desired result.  The proof is complete.
\end{proof}

The following lemma is a main result in \cite{PjRj2016jmpa} and plays an  important role in the studying of Toeplitz operators on $A_\om^2$.

\begin{Lemma}\label{1104-3}
Let $0<p<\infty$ and  $\om\in\hD$. Then the following assertions hold.
\begin{enumerate}[(i)]
  \item
     $M_p^p(r, B_z^\om) \approx \int_{0}^{r|z|} \frac{1}{\hat{\om}(t)^p (1-t)^{p}}dt,$\,\,$r|z|\to 1$.
  \item  If $\upsilon\in\hD$, $\|B_z^\om\|_{A_\upsilon^p}^p \approx   \int_{0}^{|z|} \frac{\hat{\upsilon}(t)}{\hat{\om}(t)^p (1-t)^{p}}dt$,\,\,$|z|\to 1$.
\end{enumerate}
Here and hence forth, $M^p_p(r,B_z^\om)= \frac{1}{2\pi}\int_0^{2\pi}|B_z^\om(re^{i\theta})|^pd\theta $.
\end{Lemma}

For any measurable function $f$, let
$$B_\om f(z)=\frac{1}{\|B_z^\om\|_{A_\om^2}^2}\int_{\DD} f(w)|B_z^\om(w)|^2\frac{\hat{\om}(w)dA(w)}{1-|w|}.$$
Then we have the following lemma.

\begin{Lemma}\label{1023-1} Let $\om\in\dD$ and $1\leq p<\infty$. Then there exists $\varepsilon>0$ such that
$B_\om:L^p_\tau\to L^p_\tau$ is bounded  when $\tau=-2\pm\varepsilon$.
\end{Lemma}

\begin{proof}
Since $\om\in\dD$, by Lemmas  A and B in \cite{PjaRjSk2018jga},  there are constants $0<a<b<\infty$ such that
  \begin{align}\label{0515-1}
  \frac{{\hat{\om}}(t)}{(1-t)^b} \nearrow\infty
\,\,\mbox{ and }\,\,\frac{{\hat{\om}}(t)}{(1-t)^a}\searrow 0,\,\,\mbox{ when }\,\,0\leq t<1 .
\end{align}
Suppose $0<\varepsilon<a$. Then $(1-|\cdot|)^{-1\pm\varepsilon}\hat{\om}(\cdot)\in\R$.

We only prove the case of  $\tau=-2-\varepsilon$. The case of $\tau=-2+\varepsilon$ can be proved in the same way.

Suppose $p=1$. By  Lemma \ref{1104-3},
\begin{align}
B_z^\om(z)=\|B_z^\om\|_{A_\om^2}^2 \approx\frac{1}{(1-|z|)\hat{\om}(z)},\,\, \,\, z\in\DD.\label{1025-1}
\end{align}
  For all $f\in L_{\tau}^p$, by Fubini's Theorem, (\ref{1025-1}) and Lemma \ref{1104-3} {\it (ii)}, we have
\begin{align*}
\|B_\om f\|_{L_\tau^1}
&=\int_\DD \frac{1}{\|B_z^\om\|_{A_\om^2}^2} \left|\int_\DD f(w)|B_z^\om(w)|^2\frac{\hat{\om}(w)dA(w)}{1-|w|}\right|(1-|z|^2)^{-2-\varepsilon} dA(z)   \\
&\lesssim \int_\DD |f(w)|\frac{\hat{\om}(w)dA(w)}{1-|w|}\int_\DD   |B_w^\om(z)|^2 (1-|z|)^{-1-\varepsilon} \hat{\om}(z)dA(z)\\
&\lesssim \int_\DD |f(w)|\frac{\hat{\om}(w)dA(w)}{1-|w|}\int_0^{\frac{|w|+1}{2}} \frac{1}{(1-t)^{2+\varepsilon}\hat{\om}(t)} dt.
\end{align*}
Since $0<\varepsilon<a$,
\begin{align*}
\int_0^{\frac{|w|+1}{2}} \frac{1}{(1-t)^{2+\varepsilon}\hat{\om}(t)} dt
\lesssim \frac{(1-|w|)^a}{\hat{\om}(w)}\int_0^{\frac{|w|+1}{2}} \frac{1}{(1-t)^{2+a+\varepsilon}} dt
\lesssim \frac{1}{(1-|w|)^{1+\varepsilon}\hat{\om}(w)}.
\end{align*}
So,  $\|B_\om f\|_{L_\tau^p}^p\lesssim  \| f\|_{L_\tau^p}^p$ when $p=1$.

Suppose $1<p<\infty$. Let $p^\p=\frac{p}{p-1}$, $h(z)=(1-|z|)^s$ with $0<s<\min\{\frac{a}{p^\p},\frac{a}{p}\}$. Set
$$H(z,w)=\frac{|B_z^\om(w)|^2\hat{\om}(w)}{\|B_z^\om\|_{A_\om^2}^2(1-|w|)(1-|w|^2)^{-2-\varepsilon}}.$$
Then $ B_\om f(z)=\int_\DD f(w)H(z,w)dA_{\tau}(w) $. On one hand,
\begin{align*}
\int_\DD H(z,w)h(w)^{p^\p}dA_{\tau}(w)
&=\int_\DD  \frac{|B_z^\om(w)|^2\hat{\om}(w)(1-|w|)^{p^\p s-1}}{\|B_z^\om\|_{A_\om^2}^2}dA(w)\\
&\lesssim (1-|z|)\hat{\om}(z)\int_0^\frac{|z|+1}{2}\frac{1}{(1-t)^{2-p^\p s}\hat{\om}(t)}dt  \\
&\lesssim \frac{(1-|z|)^{1+a}\hat{\om}(z)}{\hat{\om}(z)}\int_0^\frac{|z|+1}{2}\frac{1}{(1-t)^{2+a-p^\p s}}dt\\
&\lesssim h(z)^{p^\p}.
\end{align*}
On the other hand,
\begin{align*}
\int_\DD H(z,w)h(z)^pdA_{\tau}(z)
&\approx \int_\DD \frac{|B_z^\om(w)|^2\hat{\om}(w)(1-|z|)^{-2-\varepsilon+ps}}{\|B_z^\om\|_{A_\om^2}^2(1-|w|^2)^{-1-\varepsilon}}dA(z)  \\
&\approx \frac{\hat{\om}(w)}{(1-|w|)^{-1-\varepsilon}}\int_\DD |B_w^\om(z)|^2(1-|z|)^{-1-\varepsilon+ps}\hat{\om}(z)dA(z)\\
&\lesssim  \frac{\hat{\om}(w)}{(1-|w|)^{-1-\varepsilon}}\int_0^\frac{|w|+1}{2} \frac{1}{(1-t)^{2+\varepsilon-ps}\hat{\om}(t)} dt\\
&\lesssim h(w)^{p}.
\end{align*}
By Schur's test, see \cite[Theorem 3.6]{zhu} for example, we have that $B_\om:L^p_\tau\to L^p_\tau$ is bounded when $1<p<\infty$.
  The proof is complete.
\end{proof}

\begin{Lemma}\label{1025-2}
Suppose $\om\in\dD$, $1\leq p\leq \infty$ and $0\leq q\leq \infty$. Then $B_\om$ is bounded on $\mathrm{K}_q^p(\lambda)$.
\end{Lemma}

\begin{proof}
When $1\leq p<\infty$,   choose  $\varepsilon>0$ such that Lemma \ref{1023-1} holds.
For  $f\in \mathrm{K}_q^p(\lambda)$, by Lemma \ref{1023-1}, when $\tau=-2\pm\varepsilon$, for $j=0,1,2\cdots$,  we have
\begin{align*}
\|B_\om  ( f\chi_j )\|_{L_{\lambda \chi_k}^p}^p
&=\int_{1-\frac{1}{2^k}\leq |z|<1-\frac{1}{2^{k+1}}} |B_\om(f\chi_j)(z)|^p\frac{dA(z)}{(1-|z|^2)^{2}}   \\
&\approx 2^{k(\tau+2)}\int_{1-\frac{1}{2^k}\leq |z|<1-\frac{1}{2^{k+1}}} |B_\om(f\chi_j)(z)|^p dA_\tau(z)  \\
&\lesssim 2^{k(\tau+2)}\int_{\DD} |f(z)\chi_j(z)|^p dA_\tau(z)    \approx 2^{(\tau+2)(k-j)}\|f\chi_j\|_{L_\lambda^p}^p.
\end{align*}
When $j\leq k$, letting $\tau=-2-\varepsilon$, we get
$$\|B_\om  ( f\chi_j )\|_{L_{\lambda \chi_k}^p}^p\lesssim 2^{-\varepsilon |j-k|} \|f\chi_j\|_{L_\lambda^p}^p.$$
When $j>k$, letting $\tau=-2+\varepsilon$, we obtain
$$\|B_\om  ( f\chi_j )\|_{L_{\lambda \chi_k}^p}^p\lesssim 2^{-\varepsilon |j-k|} \|f\chi_j\|_{L_\lambda^p}^p.$$
Therefore, when $1\leq p<\infty$,
\begin{align}\label{1104-8}
\|(B_\om f)\chi_k\|_{L_\lambda^p}
=\left\|B_\om \left(\sum_{j=0}^\infty f\chi_j\right)\right\|_{L_{\lambda \chi_k}^p}
\leq \sum_{j=0}^\infty \|B_\om  ( f\chi_j )\|_{L_{\lambda \chi_k}^p}
\lesssim \sum_{j=0}^\infty \frac{\|f\chi_j\|_{L_\lambda^p}}{2^\frac{\varepsilon|j-k|}{p}}.
\end{align}

When $p=\infty$,
\begin{align}\label{1112-1}
\|(B_\om f)\chi_k\|_{L_\lambda^\infty}
=\|B_\om f\|_{L_{\lambda \chi_k}^\infty}
&\leq \left(\sum_{j=0}^k +\sum_{j=k+1}^\infty \right) \|B_\om  ( f\chi_j )\|_{L_{\lambda \chi_k}^\infty}.
\end{align}
After a calculation,
\begin{align}
\|B_\om  ( f\chi_j )\|_{L_{\lambda \chi_k}^\infty}
= &\sup_{1-\frac{1}{2^k}\leq |z|<1-\frac{1}{2^{k+1}}}\frac{1}{\|B_z^\om\|_{A_\om^2}^2}\int_{1-\frac{1}{2^j}\leq|w|<1-\frac{1}{2^{j+1}}}
      |f(w)||B_z^\om(w)|^2\frac{\hat{\om}(w)dA(w)}{1-|w|}  \nonumber\\
\lesssim& \|f\|_{L_{\lambda \chi_j}^\infty} \frac{1}{2^{k}}\hat{\om}(1-\frac{1}{2^k})\int_{1-\frac{1}{2^j}}^{1-\frac{1}{2^{j+1}}}\frac{\hat{\om}(r)dr}{1-r}
\int_0^{(1-\frac{1}{2^{j+1}})(1-\frac{1}{2^{k+1}})}  \frac{1}{(1-t)^2 \hat{\om}(t)^2}dt  \nonumber\\
\approx&  \|f\|_{L_{\lambda \chi_j}^\infty} \frac{1}{2^{k}}\hat{\om}(1-\frac{1}{2^{k+1}})  \hat{\om}(1-\frac{1}{2^{j+1}})
\int_0^{x_{j,k}}  \frac{1}{(1-t)^2 \hat{\om}(t)^2}dt.  \label{1104-6}
\end{align}
Here $x_{j,k}=(1-\frac{1}{2^{j+1}})(1-\frac{1}{2^{k+1}})$.

By (\ref{0515-1}),
\begin{align}
\int_0^{x_{j,k}}  \frac{1}{(1-t)^2 \hat{\om}(t)^2}dt
\lesssim&
\frac{(1-x_{j,k})^{2a} }{\hat{\om}(x_{j,k})^2}
 \int_0^{x_{j,k}}  \frac{1}{(1-t)^{2+2a}}dt  \nonumber\\
\lesssim &
\frac{(1-x_{j,k})^{-1}  }{\hat{\om}(x_{j,k})^2}
   = \frac{\left(1-(1-\frac{1}{2^{k+1}})(1-\frac{1}{2^{j+1}})\right)^{-1}  }{\left(\hat{\om}((1-\frac{1}{2^{k+1}})(1-\frac{1}{2^{j+1}}))\right)^2}.  \label{1104-7}
\end{align}
Let $r_j=\frac{1}{2^{j+1}}$. Since
$\frac{1-\left(1-\frac{1}{2^{j+1}}-\frac{1}{2^{k+1}}\right)}{1-(1-\frac{1}{2^{k+1}})(1-\frac{1}{2^{j+1}})}\approx 1,$
by  (\ref{1104-6}) and (\ref{1104-7})  we have
\begin{align*}
\|B_\om  ( f\chi_j )\|_{L_{\lambda \chi_k}^\infty}
&\lesssim \|f\|_{L_{\lambda \chi_j}^\infty}
\frac{\hat{\om}(1-\frac{1}{2^{k+1}})\hat{\om}(1-\frac{1}{2^{j+1}})}  {\left(1+2^{k-j}\right)\left(\hat{\om}(1-\frac{1}{2^{k+1}}-\frac{1}{2^{j+1}})\right)^2}\\
&=\|f\|_{L_{\lambda \chi_j}^\infty}
\frac{\hat{\om}(1-r_k)\hat{\om}(1-r_j)}  {\left(1+2^{k-j}\right)\hat{\om}(1-r_k-r_j)^2}.
\end{align*}
Therefore, when $j\geq k$, using the monotonicity of $\hat{\om}$, we have
\begin{align}
\|B_\om ( f\chi_j )\|_{L_{\lambda \chi_k}^\infty}
&\lesssim \|f\|_{L_{\lambda \chi_j}^\infty}\frac{\hat{\om}(1-r_j)}  {\hat{\om}(1-r_k-r_j)}\nonumber\\
&=\|f\|_{L_{\lambda \chi_j}^\infty}
\frac{\frac{\hat{\om}(1-r_j)}{\left(1-(1-r_j)\right)^{a}}}
{\frac{\hat{\om}(1-r_k-r_j)}{\left(1-(1-r_k-r_j)\right)^{a}}}
\frac{r_j^{a}}{(r_k+r_j)^{a}}
\lesssim \frac{\|f\|_{L_{\lambda \chi_j}^\infty}}{2^{a(j-k)}}.\label{1031-1}
\end{align}
When $j<k$, it is obvious that
\begin{align}\label{1031-2}
\|B_\om  ( f\chi_j )\|_{L_{\lambda \chi_k}^\infty}\lesssim \frac{\|f\|_{L_{\lambda \chi_j}^\infty}}{2^{k-j}}.
\end{align}
From  (\ref{1112-1}), (\ref{1031-1}) and (\ref{1031-2}), we have
\begin{align}\label{1112-2}
\|(B_\om f)\chi_k\|_{L_\lambda^\infty}
\leq \sum_{j=0}^\infty  \frac{\|f\chi_j\|_{L_\lambda^\infty}}{2^{{\min\{1,a\}}|j-k|}}.
\end{align}

By (\ref{1104-8}) and (\ref{1112-2}),  there exists $\varepsilon_0>0$ such that
\begin{align*}
\|(B_\om f)\chi_k\|_{L_\lambda^p}
&\lesssim
\sum_{j=0}^\infty \frac{\|f\chi_j\|_{L_\lambda^p}}{2^{\varepsilon_0|j-k|}},
\end{align*}
when $1\leq p\leq \infty$.

When $1\leq q\leq \infty$, we consider the sequences $X=\{x_k\}$, $Y=\{y_k\}$,  where $x_k=2^{-\varepsilon_0 |k|}$ and
$$y_k=\left\{
\begin{array}{cc}
  \|f\chi_k\|_{L_\lambda^p}, & k\geq 0, \\
  0, & k<0.
\end{array}
\right.
$$
Then $\|(B_\om f)\chi_k\|_{L_\lambda^p}\lesssim X*Y(k)$. By Young's inequality,
\begin{align*}\|B_\om f\|_{\mathrm{K}_q^p(\lambda)}
&=\|\{\|(B_\om f)\chi_k\|_{L^p_{\lambda}}\}_{k=0}^\infty\|_{l^q}  \\
&\lesssim  \left\|\left\{   \sum_{j=0}^\infty \frac{\|f\chi_j\|_{L_\lambda^p}}{2^{\varepsilon_0|j-k|}}   \right\}_{k=0}^\infty\right\|_{l^q}
=\|X*Y\|_{l^q}\lesssim \|X\|_{l^1}\|Y\|_{l^q}\approx \|f\|_{\mathrm{K}_q^p(\lambda)}.
\end{align*}

When $0<q<1$,
\begin{align*}
\|B_\om f\|_{K_q^p(\lambda)}^q
&=\sum_{k=0}^\infty\left( \|(B_\om f)\chi_k\|_{L_\lambda^p} \right)^q
 \lesssim \sum_{k=0}^\infty\left( \sum_{j=0}^\infty \frac{\|f\chi_j\|_{L_\lambda^p}}{2^{\varepsilon_0|j-k|}}\right)^q\\
 &\leq \sum_{k=0}^\infty \sum_{j=0}^\infty \frac{\|f\chi_j\|_{L_\lambda^p}^q}{2^{q\varepsilon_0|j-k|}} \leq \sum_{j=0}^\infty \|f\chi_j\|_{L^p_\lambda}^q  \sum_{k=-\infty}^\infty 2^{-q\varepsilon_0 |k|} \approx \|f\|_{\mathrm{K}_q^p(\lambda)}^q.
\end{align*}

If $q=0$, for any given $M\in\N$,
\begin{align*}
\|B_\om f\|_{K_q^p(\lambda)}
&=\limsup_{k\to\infty}\|(B_\om f)\chi_k\|_{L_\lambda^p}
\lesssim \limsup_{k\to\infty} \sum_{j=0}^\infty \frac{\|f\chi_j\|_{L_\lambda^p}}{2^{\varepsilon_0|j-k|}}   \\
&= \limsup_{k\to\infty} \left(\sum_{j=0}^M \frac{\|f\chi_j\|_{L_\lambda^p}}{2^{\varepsilon_0|j-k|}}
+ \sum_{j=M+1}^\infty \frac{\|f\chi_j\|_{L_\lambda^p}}{2^{\varepsilon_0|j-k|}}\right)
\lesssim \sup_{j\geq M+1} \|f\chi_j\|_{L_\lambda^p}.
\end{align*}
Letting $M\to\infty$, we have  $\|B_\om f\|_{K_q^p(\lambda)}\lesssim \|f\|_{K_q^p(\lambda)}$.
The proof is complete.
\end{proof}

\noindent{\bf Proof of Theorem \ref{1103-1}.}  ({\it i}).
By Lemma \ref{1105-1}, for any $z\in \Lambda_k$,  there exists $N=N(r)$ such that  $D(z,r)\subset\cup_{j=k-N}^{k+N} \Lambda_j$. Here  $\Lambda_j=\O$ if $j<0$.
Thus,
$$(\widehat{\mu_r}\chi_k)(z)\leq \sum_{j=k-N}^{k+N}\frac{\mu(\Lambda_j\cap D(z,r))}{(1-|z|)\hat{\om}(z)}
=\sum_{j=k-N}^{k+N}\widehat{(\mu\chi_j)_r}(z) $$
and
$$
\widehat{(\mu\chi_j)_r}(z)=\frac{\mu(\Lambda_j\cap D(z,r))}{(1-|z|)\hat{\om}(z)}\leq \frac{\mu(D(z,r))}{(1-|z|)\hat{\om}(z)}\sum_{k=j-N}^{j+N}\chi_k(z)=\sum_{k=j-N}^{j+N}(\widehat{\mu_r}\chi_k)(z).
$$
Therefore, by Theorem 3 in \cite{PjaRjSk2018jga},
\begin{align}
\|\T_\mu\|_{\mathcal{S}_{p,q}}
=\|\{\|\T_{\mu\chi_k}\|_{\mathcal{S}_p}\}_{k=0}^\infty\|_{l^q}
 \approx \|\{\|\widehat{(\mu\chi_k)_r}\|_{L^p_\lambda}\}_{k=0}^\infty\|_{l^q}    \approx \|\{\|\widehat{\mu_r}\chi_k\|_{L^p_\lambda}\}_{k=0}^\infty\|_{l^q}
=\|\widehat{\mu_r}\|_{\mathrm{K}_q^p(\lambda)}    \nonumber
\end{align}
and
\begin{align}
\|\T_\mu\|_{\mathcal{S}_{p,q}}
=\|\{\|\T_{\mu\chi_k}\|_{\mathcal{S}_p}\}_{k=0}^\infty\|_{l^q}
&\approx \|\{\|\{\widehat{(\mu\chi_k)_r}(a_j)\}_{j=1}^\infty\|_{l^p}\}_{k=0}^\infty\|_{l^q}   \nonumber\\
&\approx \|\{\|\{(\widehat{\mu_r}\chi_k)(a_j)\}_{j=1}^\infty\|_{l^p}\}_{k=0}^\infty\|_{l^q}
=\|\{\widehat{\mu_r}(a_j)\}\|_{l_q^p}.   \nonumber
\end{align}

({\it ii}). Suppose $r>0$ such that Lemma 8 in \cite{PjaRjSk2018jga} holds. That is,  for all $z\in\DD$ and $w\in D(z,r)$, $|B_z^\om(w)|\approx B_z^\om(z)$.
By Lemma 11 in \cite{PjaRjSk2018jga},
$$\langle \T_\mu B_z^\om, B_z^\om \rangle_{A_\om^2}=\langle B_z^\om,B_z^\om\rangle_{L_\mu^2}.$$
Using (\ref{1025-1}), we get
$$\widehat{\mu_r}(z)=\frac{\mu(D(z,r))}{(1-|z|)\hat{\om}(z)}
\approx\frac{1}{B_z^\om(z)}\int_{D(z,r)} |B_z^\om(w)|^2d\mu(w)\leq \widetilde{\T_\mu}(z).
$$
So, $$\|\T_\mu\|_{\mathcal{S}_{p,q}}\approx\|\widehat{\mu_r}\|_{\mathrm{K}_q^p(\lambda)}\lesssim \|\widetilde{\T_\mu}\|_{\mathrm{K}_q^p(\lambda)}.$$

On the other hand,  since $|B_z^\om(w)|^2$ is subharmonic and $\om\in\dD$,  by Fubini's theorem,
\begin{align*}
\widetilde{\T_\mu}(z)
&=\frac{1}{B_z^\om(z)}\int_\DD |B_z^\om(w)|^2d\mu(w)    \lesssim \frac{1}{B_z^\om(z)}\int_\DD \int_{D(w,r)}\frac{|B_z^\om(\eta)|^2}{(1-|\eta|)^2}dA(\eta)d\mu(w)  \\
&=\frac{1}{B_z^\om(z)}\int_\DD \widehat{\mu_r}(\eta) |B_z^\om(\eta)|^2\frac{\hat{\om}(\eta)dA(\eta)}{1-|\eta|} =B_\om(\widehat{\mu_r})(z).
\end{align*}
By Lemma \ref{1025-2} we have
$$\|\widetilde{\T_\mu}\|_{\mathrm{K}_q^p(\lambda)}
\lesssim \|B_\om(\widehat{\mu_r})\|_{\mathrm{K}_q^p(\lambda)}
\lesssim \|\widehat{\mu_r}\|_{\mathrm{K}_q^p(\lambda)}
\approx\|\T_\mu\|_{\mathcal{S}_{p,q}},$$
when $1\leq p\leq\infty$.

Next we prove the case of $p<1$.  For $z\in \DD$,
\begin{align*}
\widetilde{\T_\mu}(z)
&=\frac{1}{B_z^\om(z)}\int_\DD |B_z^\om(w)|^2d\mu(w)
\leq \frac{1}{B_z^\om(z)}\sum_{j=1}^\infty\int_{D(a_j,r)}|B_z^\om(w)|^2d\mu(w)\\
&\leq \frac{1}{B_z^\om(z)}\sum_{j=1}^\infty\mu(D(a_j,r))\sup_{w\in\ol{D(a_j,r)}}|B_z^\om(w)|^2.
\end{align*}
Using the subharmonicity  of $|B_z^\om|^{2p}$,
\begin{align*}
\big(\widetilde{\T_\mu}(z)\big)^p
\leq& \frac{1}{(B_z^\om(z))^p}\sum_{j=1}^\infty(\mu(D(a_j,r)))^p\sup_{w\in\ol{D(a_j,r)}}|B_z^\om(w)|^{2p}\\
\lesssim &\frac{1}{(B_z^\om(z))^p}\sum_{j=1}^\infty\left(\widehat{\mu_r}(a_j)\right)^p  \frac{(1-|a_j|)^{p}\hat{\om}(a_j)^p}{(1-|a_j|)^2}
\int_{D(a_j,2r)}|B_z^\om(w)|^{2p}dA(w).
\end{align*}
By Fubini's theorem and Lemma \ref{1104-3} {\it (i)},
\begin{align*}
 \int_{\Lambda_k}\int_{D(a_j,2r)}|B_z^\om(w)|^{2p}dA(w)d\lambda(z)
\approx&2^{2k}\int_{D(a_j,2r)}\int_{\Lambda_k}|B_z^\om(w)|^{2p}dA(z)dA(w)\\
\lesssim &2^k \int_{D(a_j,2r)}\int_0^{(1-\frac{1}{2^{k+1}})\frac{|w|+1}{2}}\frac{ dtdA(w)}{\hat{\om}(t)^{2p}(1-t)^{2p}}.
\end{align*}
Let $\upsilon(z)=\frac{(1-|z|)^{p}\hat{\om}(z)^p}{(1-|z|)^2}$.
Since $\upsilon\in\dD$, there exist $C_1,C_2>1$ and $K>1$, such that
$$\hat{\upsilon}(t)
=\int_t^\frac{1+t}{2}\upsilon(s)ds+\hat{\upsilon}(\frac{1+t}{2})
>\int_t^\frac{1+t}{2}\upsilon(s)ds+\frac{1}{C_1}\hat{\upsilon}(t) $$
and
$$\hat{\upsilon}(t)
=\int_t^{1-\frac{1-t}{K}}\upsilon(s)ds+\hat{\upsilon}(1-\frac{1-t}{K})
<\int_t^{1-\frac{1-t}{K}}\upsilon(s)ds+\frac{1}{C_2}\hat{\upsilon}(t).$$
So,
$$(1-t)\upsilon(t)\approx \int_t^\frac{1+t}{2}\upsilon(s)ds
\lesssim\hat{\upsilon}(t)\lesssim
\int_t^{1-\frac{1-t}{K}}\upsilon(s)ds\approx (1-t)\upsilon(t),
$$
which implies that $\upsilon\in\R$. Therefore,
\begin{align*}
\int_0^{(1-\frac{1}{2^{k+1}})\frac{|w|+1}{2}}\frac{1}{\hat{\om}(t)^{2p}(1-t)^{2p}} dt
&\approx \int_0^{(1-\frac{1}{2^{k+1}})\frac{|w|+1}{2}}\frac{1}{{\hat{\upsilon}}(t)^2(1-t)^2} dt\\
&\lesssim \frac{1}{\left(\hat{\upsilon}((1-\frac{1}{2^{k+1}})\frac{|w|+1}{2})\right)^2\Big(1-(1-\frac{1}{2^{k+1}})\frac{|w|+1}{2}\Big)}.
\end{align*}
Since $w\in D(a_j,r)$,  we get
\begin{align*}
1-(1-\frac{1}{2^{k+1}})\frac{|a_j|+1}{2}  \approx  1-(1-\frac{1}{2^{k+1}})\frac{|w|+1}{2}.
\end{align*}
Hence,
$$ \int_{\Lambda_k}\int_{D(a_j,2r)}|B_z^\om(w)|^{2p}dA(w)d\lambda(z)
\lesssim  \frac{2^k(1-|a_j|)^2}{\left(\hat{\upsilon}((1-\frac{1}{2^{k+1}})\frac{|a_j|+1}{2})\right)^2\Big(1-(1-\frac{1}{2^{k+1}})\frac{|a_j|+1}{2}\Big)}.$$
Let $\xi_m=\left(\sum_{a_j\in \Lambda_m}(\widehat{\mu_r}(a_j))^p\right)^\frac{1}{p}$. Then
\begin{align*}
&\int_{\Lambda_k}(\widetilde{\T_\mu}(z))^p d\lambda(z)\\
\lesssim&  \frac{1}{2^{(k+1)p}}\hat\om(1-\frac{1}{2^{k+1}})^p\sum_{m=0}^\infty  \xi_m^p
\frac{\frac{1}{2^{(m+1)p}}\hat{\om}(1-\frac{1}{2^{m+1}})^p}{\frac{1}{2^{2(m+1)}}}
\frac{\frac{2^{k}}{2^{2(m+1)}}}{\left(\hat{\upsilon}(1-\frac{1}{2^{k+1}}-\frac{1}{2^{m+1}})\right)^2(\frac{1}{2^{k+1}}+\frac{1}{2^{m+1}})}\\
\approx
&\sum_{m=0}^\infty\xi_m^p\frac{\hat{\upsilon}(1-\frac{1}{2^{m+1}})\hat{\upsilon}(1-\frac{1}{2^{k+1}})}
{\hat{\upsilon}(1-\frac{1}{2^{m+1}}-\frac{1}{2^{k+1}})^2}
\frac{1}{1+2^{m-k}}.
\end{align*}
Similarly to get (\ref{1031-1}) and (\ref{1031-2}), there exists $\varepsilon>0$ such that
$$\int_{\Lambda_k}\big(\widetilde{\T_\mu}(z)\big)^p d\lambda(z)\lesssim \sum_{m=0}^\infty\xi_m^p \cdot\frac{1}{2^{\varepsilon|k-m|}}.$$
Note that
$$\|\{\xi_m\}\|_{l^q}
=\left\|\left\{\left(\sum_{j=1}^\infty ((\widehat{\mu_r}\chi_m)(a_j))^p\right)^\frac{1}{p}\right\}_{m=0}^\infty\right\|_{l^q}
=\|\{\widehat{\mu_r}(a_j)\}\|_{l^p_q}
\approx \|\T_\mu\|_{{\mathcal{S}}_{p,q}}.$$

If $0<p<q=\infty$, then
\begin{align*}
\|\widetilde{\T_\mu}\|_{\mathrm{K}_\infty^p(\lambda)}=\sup_{k\geq 0} \left(\int_{\Lambda_k}\big(\widetilde{\T_\mu}(z)\big)^p d\lambda(z)\right)^\frac{1}{p}
\lesssim \|\{\xi_m\}\|_{l^\infty}
\approx \|\T_\mu\|_{{\mathcal{S}}_{p,\infty}}.
\end{align*}

When $0<p<q<\infty$, let $X=\{x_k\}$ and $Y=\{y_k\}$, where $x_k=\frac{1}{2^{\varepsilon|k|}}$ and
$$y_k=\left\{
\begin{array}{cc}
\xi_k^p,   & k\geq 0, \\
 0, & k<0.
\end{array}
\right.$$
Then,
\begin{align*}
\|\widetilde{\T_\mu}\|_{\mathrm{K}_q^p(\lambda)}
&=\left(\sum_{k=0}^\infty  \left(\int_{\Lambda_k}\big(\widetilde{\mu}(z)\big)^p d\lambda(z)\right)^\frac{q}{p}\right)^\frac{1}{q}
\lesssim\left(\sum_{k=0}^\infty  \left(\sum_{m=0}^\infty\xi_m^p \cdot\frac{1}{2^{\varepsilon|k-m|}}\right)^\frac{q}{p}\right)^\frac{1}{q}  \\
&=\|X*Y\|_{l^\frac{q}{p}}^\frac{1}{p}  \leq \left(\|X\|_{l^1}\|Y\|_{l^\frac{q}{p}}\right)^\frac{1}{p} \lesssim \|\xi_m\|_{l^q}\approx \|\T_\mu\|_{\mathcal{S}_{p,q}}.
\end{align*}

If $0<q\leq p<1$, then
\begin{align*}
\|\widetilde{\T_\mu}\|_{\mathrm{K}_q^p(\lambda)}
&=\left(\sum_{k=0}^\infty  \left(\int_{\Lambda_k}\big(\widetilde{\T_\mu}(z)\big)^p d\lambda(z)\right)^\frac{q}{p}\right)^\frac{1}{q}
\lesssim\left(\sum_{k=0}^\infty  \left(\sum_{m=0}^\infty\xi_m^p \cdot\frac{1}{2^{\varepsilon|k-m|}}\right)^\frac{q}{p}\right)^\frac{1}{q}  \\
&\lesssim\left(\sum_{k=0}^\infty  \sum_{m=0}^\infty\xi_m^q \cdot\frac{1}{2^{\frac{q\varepsilon|k-m|}{p}}}\right)^\frac{1}{q}
\lesssim\left(\sum_{m=0}^\infty\xi_m^q\sum_{k=-\infty}^\infty   \cdot\frac{1}{2^{\frac{q\varepsilon|k|}{p}}}\right)^\frac{1}{q} \lesssim \|\xi_m\|_{l^q}\approx \|\T_\mu\|_{\mathcal{S}_{p,q}}.
\end{align*}

When $0=q<p<1$, for any given $M\in\N$,
\begin{align*}
\limsup_{k\to\infty}\int_{\Lambda_k}\big(\widetilde{\T_\mu}(z)\big)^p d\lambda(z)
&\lesssim \limsup_{k\to\infty}\left(\sum_{m=0}^M+\sum_{m=M+1}^\infty\right)\xi_m^p \cdot\frac{1}{2^{\varepsilon|k-m|}}
\leq \sup_{m>M}\xi_m^p.
\end{align*}
Letting $M\to\infty$, we obtain
$$\limsup_{k\to\infty}\left(\int_{\Lambda_k}\big(\widetilde{\T_\mu}(z)\big)^p d\lambda(z)\right)^\frac{1}{p}\lesssim \limsup_{m\to\infty}\xi_m.$$
So, $$\|\widetilde{\T_\mu}\|_{\mathrm{K}_0^p(\lambda)}\lesssim \|\xi_m\|_{l^0}\approx\|\T_\mu\|_{\mathcal{S}_{p,q}}.$$
The proof is complete.\hfill\hfil$\square$

\end{document}